\documentclass[11pt]{amsart}
\usepackage{amsmath,amsthm,amsmath,amsrefs}

\usepackage{amsbsy,amsmath,amsfonts,amsthm,amssymb,color}

\theoremstyle{plain}
\newtheorem{mydef}{Definition}
\newtheorem{mythm}{Theorem}[section]

\newtheorem{lemma}{Lemma}[section]

\newtheorem{prop}{Proposition}[section]

\newtheorem{quest}{Question}

\usepackage[pdftex]{graphicx}

\usepackage{mathrsfs}

\newcommand{\mc}{\mathcal}

\newcommand{\N}{\mathbb{N}}

\newdimen\Squaresize \Squaresize=14pt
\newdimen\Thickness \Thickness=0.4pt

\def\Square#1{\hbox{\vrule width \Thickness
   \vbox to \Squaresize{\hrule height \Thickness\vss
      \hbox to \Squaresize{\hss#1\hss}
   \vss\hrule height\Thickness}
\unskip\vrule width \Thickness} \kern-\Thickness}

\def\Vsquare#1{\vbox{\Square{$#1$}}\kern-\Thickness}

\title[Iterative algebras]{Iterative algebras}

\author{Jason P. Bell and Blake W. Madill}
\thanks{The authors thank NSERC for its generous support.}
\address{University of Waterloo \\
Department of Pure Mathematics \\
Waterloo, Ontario \\
Canada  N2L 3G1\\}
\email{jpbell@uwaterloo.ca}
\email{bmadill@uwaterloo.ca}



\begin{document}

\begin{abstract} Given a finitely generated free monoid $X$ and a morphism $\phi : X\to X$, we show that one can construct an algebra, which we call an iterative algebra, in a natural way.  We show that many ring theoretic properties of iterative algebras can be easily characterized in terms of linear algebra and combinatorial data from the morphism and that, moreover, it is decidable whether or not an iterative algebra has these properties.  Finally, we use our construction to answer several questions of Greenfeld, Leroy,
Smoktunowicz, and Ziembowski by constructing a primitive graded nilpotent algebra with Gelfand-Kirillov dimension two that is finitely generated as a Lie algebra.
\end{abstract}

\maketitle
\section{Introduction}
A rich class of examples in ring theory is provided by monomial algebras, which have the advantage of having the property that many ring theoretic properties can be understood purely combinatorially in terms of forbidden subwords and other notions from combinatorics on words (cf. \cite{Belov}).  In addition, the study of monomial algebras via combinatorial methods has played a key role in many results from ring theory dealing with growth and other properties due to the intimate connection with noncommutative Gr\"obner bases \cite{Drensky, Giambruno}.  In this paper, we study a certain subclass of monomial algebras which we call \emph{iterative algebras} (see \S 2 for the definition).  One of the problems with the class of monomial algebras is that it is in general a very large and unwieldy class.  One can, in some settings, restrict one's attention to finitely presented monomial algebras, and in this case one has graph theoretic and other combinatorial tools at one's disposal for their study 
(cf. \cite[Chapter 5]{Belov}).  On the other hand, this class of algebras can often be too well-behaved to give one a sense of the types of pathologies that can occur in noncommutative algebra.  Iterative algebras are, in general, non-finitely presented monomial algebras produced via iterated morphisms of monoids.  They retain many of the nice combinatorial properties of finitely presented algebras, but at the same time can sometimes possess bizarre properties due to the iterative nature of their definition.  

Informally, an iterative algebra is obtained by taking a field $k$ and then forming the quotient of $k\{x_1,\ldots ,x_d\}/I$, where $I$ is the ideal generated by all words over the alphabet $\{x_1,\ldots ,x_d\}$ that do not occur as a subword of a right-infinite word $w$ over the alphabet $\{x_1,\ldots ,x_d\}$ that is a fixed point of an endomorphism $\phi$ of the free monoid on $\{x_1,\ldots ,x_d\}$.  In particular, the right-infinite word $w$, being a fixed point of an endomorphism, has many self-similarity properties that make studying the structure theory of the iterative algebras we construct fairly straightforward.  In fact, we are able to show that ring theoretic properties such as begin prime, semiprime, satisfying a polynomial identity and being noetherian are all decidable in terms of the combinatorics of the endomorphism $\phi$ (see Theorem \ref{thm: char} and \S 5).  We are also able to show that the algebras we construct have Gelfand-Kirillov dimension in the set $\{1,2,3\}$ (see Theorem \ref{thm: GK}).  The impetus for the construction of these algebras, however, came from our attempt to provide answers to some questions of Greenfeld, Leroy,
Smoktunowicz, and Ziembowski \cite{GLSZ}.  These questions have to do with graded nilpotent algebras.  These are graded algebras $A$ with the property that if $S$ is a subset of $A$ consisting of homogeneous elements of the same degree, then there is some natural number $N=N(S)$ such that $S^N=(0)$.  Using our iterative algebra construction, we are able to present an example or a ring that simultaneously provides  answers to three questions of Greenfeld \emph{et al.} \cite{GLSZ}.
 \begin{mythm}\label{thm: main} Let $k$ be a field.  Then there exists a finitely generated $k$-algebra $R$ that is a prime, graded nilpotent algebra of Gelfand-Kirillov dimension two that has trivial Jacobson radical and is finitely generated as a Lie algebra.\end{mythm}
This paper is organized as follows.  In \S2, we describe morphic words and give the construction of an iterative algebra.  In \S3, we show that the Gelfand-Kirillov dimension of these algebras is either one, two, or three.  In \S3, we give characterizations of various ring theoretic properties in terms of combinatorial data of the underlying morphism and associated infinite word; then in \S4, we discuss decidability of these properties.  Finally, in \S5 we prove Theorem \ref{thm: main} (see Theorem \ref{thm: main2}).  We conclude with \S6 by giving some open problems and making some remarks.
\section{Pure morphic words and iterative algebras}
In this section, we give our construction of iterative algebras and give some of the necessary background on morphic words.

Let $\Sigma$ be a finite alphabet.  We let $\Sigma^*$ denote the free monoid on $\Sigma$ (we let $\varepsilon$ denote the empty word, which serves as the identity of this monoid).  Then a morphism $\phi :\Sigma^*\to \Sigma^*$ is determined by the images of the elements of $\Sigma$.   We say that $a\in\Sigma$ is \emph{mortal} if there is some $j$ such that $\phi^j(a)=\varepsilon$, and we let $X$ denote the submonoid of $\Sigma^*$ generated by the mortal letters.  

We then say that the morphism $\phi$ is \emph{prolongable} on the letter $b\in \Sigma$ if $\phi(b)=bx$ with $x\in \Sigma^*\setminus X$.   In this case, we can create an infinite sequence 
$$b,\phi(b)=bx,\phi^2(b)=bx\phi(x), \phi^3(b)=bx\phi(x)\phi^2(x),\ldots $$
by repeatedly iterating $\phi$ and starting with the letter $b$.  Notice that the limit of this sequence is an infinite word
$$\phi^{\omega}(b):=bx\phi(x)\phi^2(x)\cdots.$$  
We then say that a right-infinite word $w$ over the alphabet $\Sigma$ is a \emph{pure morphic word} if there exists a morphism $\phi:\Sigma^*\to \Sigma^*$ that is prolongable on some letter $b\in \Sigma$ such that $w=\phi^{\omega}(b)$.  In general, a \emph{morphic word} is obtained from a pure morphic word by applying a \emph{coding}; that is a letter-by-letter substitution of $w$ by a map $\tau: \Sigma\to \Delta$ where $\Delta$ is another finite alphabet and $\tau$ is not necessarily one-to-one.  For our purposes, we only work with pure morphic words.

We now introduce some terminology, which will give important classes of morphic words.
\begin{mydef}{\em
Let $w$ be a right-infinite morphic word that is associated to a morphism $\phi:\Sigma^*\to \Sigma^*$ that is prolongable on $b\in \Sigma$ for some finite alphabet $\Sigma$.  \begin{enumerate}
\item[(i)] We say that $w$ is a $d$-\emph{uniform} morphic word if $\phi(a)$ is a word in $\Sigma^*$ of length $d$ for every $a\in \Sigma$. 
\item[(ii)] We say that $w$ is \emph{primitive} if for every $a,a'\in \Sigma$ we have that there is some natural number $m$, depending on $a$ and $a'$, such that $a'$ appears as a letter in $\phi^m(a)$.  
\end{enumerate}}
\end{mydef}
We point out that primitivity of a pure morphic word is decidable (see \S 5 for more details).   First, observe that by removing letters from $\Sigma$ if necessary, we may always assume without any loss of generality that every $a\in \Sigma$ occurs in $w$.  Then if $\phi$ is prolongable on $b\in \Sigma$ and $w=\phi^{\omega}(b)$ then if $a'\in \Sigma$ then $a'$ appears in $\phi^n(b)$ for some $n$.  Hence if $b$ occurs as a letter in $\phi^m(a)$ for some $m$ then $a'$ will occur as a letter in $\phi^{n+m}(a)$.  Thus to check primitivity, it suffices to check that for every $a\in \Sigma$, there is some $n$, depending upon $a$, such that $b$ occurs in $\phi^n(a)$.  
We are now able to define an iterative algebra.  (We do not use the term morphic algebra, as that term has been used to name an unrelated class of rings)

Given a pure morphic word on an alphabet $\Sigma=\{x_1,\ldots ,x_m\}$, we can associate an $m\times m$ matrix, $M(w)$, called the \emph{incidence matrix}.  The $(i,j)$-entry of $M(w)$ is the number of occurrences of $x_i$ in $\phi(x_j)$.  Given a word $u\in \Sigma^*$, we can then associate an $m\times 1$ integer vector $\theta(u)$ whose $j$-th coordinate is the number of occurrences of $x_j$ in $u$.  Then \cite[Proposition 8.2.2]{AS} shows that we have the relationship
\begin{equation}
\theta(\phi^n(u)) = M(w)^n \theta(u).\label{eq: theta}
\end{equation}

\begin{mydef}{\em 
Let $k$ be a field, let $\Sigma=\{x_1,\ldots ,x_m\}$ be a finite alphabet, and let $w$ be a right-infinite pure morphic word over $\Sigma$.  We define the \emph{iterative algebra} $A_w$ associated to $w$ to be the quotient 
$k\{x_1,\ldots ,x_m\}/I$, where $I$ is the ideal generated by all finite words over $\Sigma$ that do not appear as a subword of $w$.}
\end{mydef} 
These algebras form a subclass of the monomial algebras studied in \cite[Chapter 3]{Belov}.

\section{Gelfand-Kirillov dimension}
In this section, we discuss the growth of iterative algebras, showing that their Gelfand-Kirillov dimension is either $1$, $2$, or $3$.  We recall that given two maps $f,g :\mathbb{N}_0\to \mathbb{R}_+$, we say that $f(n)=\Theta(g(n))$ if there exist positive constants $C_1$ and $C_2$ such that 
$$C_1g(n) \le f(n) \le C_2 g(n)$$ for all $n$ sufficiently large.

Given a field $k$ and a finitely generated $k$-algebra $A$, we recall that the \emph{Gelfand-Kirillov} dimension of $A$, denoted ${\rm GKdim}(A)$, is given by
$${\rm GKdim}(A):=\limsup_{n\to \infty} \frac{\log\,{\rm dim}(V^n)}{\log\, n},$$
where $V$ is a finite-dimensional subspace of $A$ that contains $1$ and generates $A$ as a $k$-algebra.  This quantity does not depend upon the choice of subspace $V$ having these properties.  For more information on Gelfand-Kirillov dimension, we refer the reader to the book of Krause and Lenagan \cite{KL}.
We note that GK dimension one is equivalent to ${\rm dim}(V^n)=\Theta(n)$ (i.e., linear growth) by Bergman's gap theorem \cite[Theorem 2.5]{KL} and that an important subclass of algebras of GK dimension two are given by those of \emph{quadratic growth}; these are algebras for which ${\rm dim}(V^n)=\Theta(n^2)$.  

In the case of iterative algebras, there is an intimate connection between the Gelfand-Kirillov dimension of the algebra and the \emph{subword complexity} (also called factor complexity in the literature) of the associated pure morphic word $w$.  We recall that given a right-infinite word $w$ over a finite alphabet $\Sigma$, the \emph{subword complexity function} is defined by taking
$p_w(n)$ to be the number of distinct subwords of $w$ of length $n$.  

For a pure morphic word $w$, a result of Pansiot \cite[Theorem 4.7.1]{BR} shows that 
$p_w(n)$ is either $O(1), \Theta(n),\Theta(n \log\log n), \Theta(n\log n)$, or $ \Theta(n^2)$; moreover, each of these possibilities can be realized by some purely morphic word and $p_w(n)=O(1)$ if and only if $w$ is eventually periodic.  If $A_w$ is the iterative algebra corresponding to the purely morphic word $w$ then 
if $V$ is the image of the space $k+\sum_{a\in \Sigma} ka$ in $A_w$ then $V^n$ has a basis consisting of subwords of $w$ of length at most $n$.  In particular,
${\rm dim}(V^n)=\sum_{j=0}^n p_w(j)$.  If $p_w(n)=\Theta(n), \Theta(n\log n),$ or $\Theta(n \log\log n)$, then we respectively have $\sum_{j=0}^n p_w(j)=\Theta(n^2), \Theta(n^2\log n),$ or $\Theta(n^2 \log\log n)$, and so
$$\frac{\log\,{\rm dim}(V^n)}{\log\, n}\to 2$$ as $n\to \infty$.  On the other hand, if $p_w(n)=\Theta(n^2)$ then
$$\frac{\log\,{\rm dim}(V^n)}{\log\, n}\to 3$$ as $n\to \infty$.  Finally, if $p_w(n)={\rm O}(1)$ then ${\rm dim}(V^n)$ grows at most linearly with $n$ and since $A_w$ is infinite-dimensional, we see that the GK dimension of $A_w$ is one in this case.  Putting these observations together we obtain the following result.
\begin{mythm} \label{thm: GK}
Let $k$ be a field and let $A_w$ be an iterative $k$-algebra.  Then ${\rm GKdim}(A_w)\in \{1,2,3\}$ and the Gelfand-Kirillov dimension is equal to one if and only if $w$ is eventually periodic.
\end{mythm}
We point out that in the case that $w$ is either primitive or $d$-uniform for some $d$, we can say more.
\begin{prop} Let $k$ be a field and let $w$ be a pure morphic word that is either primitive or $d$-uniform for some $d\ge 2$.  Then $A_w$ has either linear or quadratic growth and if $w$ is not eventually periodic then the growth is quadratic.  
\label{prop: quadratic}
\end{prop}

\begin{proof}
By \cite[Theorem 10.4.12 and Corollary 10.3.2]{AS} we see that if $w$ is primitive or $d$-uniform\footnote{Corollary 10.3.2 of \cite{AS} refers to the subword complexity of automatic words.  A result of Cobham (see \cite[Theorem 6.3.2]{AS}) gives that a $d$-uniform pure morphic word is automatic.} and $w$ is not eventually periodic then $p_w(n)=\Theta(n)$ and so the result now follows from the remarks made earlier in this section.  If $w$ is eventually periodic, then $p_w(n)=O(1)$ and so $A_w$ has linear growth.
\end{proof}
We note that the prime spectra of monomial algebras of quadratic growth are particularly well-behaved \cite{BS}. 
\section{Ring theoretic properties in terms of combinatorics of words}
In this section, we give characterizations of some basic ring theoretic properties of iterative algebras in terms of the associated morphic words.
\begin{mythm}
Let $k$ be a field, let $w$ be a right-infinite morphic word over an alphabet $\Sigma=\{x_1,\ldots ,x_m\}$, and let $A_w$ be the iterative $k$-algebra associated to $w$.  We assume that every letter in $\Sigma$ appears as a letter in $w$.  Then the following characterizations hold:
\begin{enumerate}
\item $A_w$ is prime if and only if the first letter of $w$ occurs at least twice;
\item $A_w$ is semiprime if and only if it is prime;

\item $A_w$ is just infinite if and only if $w$ is uniformly recurrent (that is, given a subword $v$ of $w$ there exists some $N=N(v)$ such that every block of $N$ consecutive letters in $w$ contains $v$ as a subword);
\item $A_w$ satisfies a polynomial identity if and only if $w$ is eventually periodic;
\item $A_w$ is noetherian if and only if $w$ is eventually periodic;

\end{enumerate}
\label{thm: char}
\end{mythm}
\begin{proof}
Since $A_w$ is a monomial algebra, we know that if $A_w$ fails to be prime then there must exist subwords $v,v'$ of $w$ such that $vuv'$ is not a subword of $w$ for any $u\in \Sigma^*$.  Thus we can never have $v'$ occurring after $v$ in $w$ and since both $v'$ and $v$ occur as subwords of $w$ we see that there is some subword $u$ of $w$ that has $v'$ as a prefix and $v$ as a suffix.  By assumption $w=\phi^{\omega}(b)$ for some $b\in \Sigma$ such that $\phi$ is prolongable on $b$, and so there is some $n$ such that $\phi^n(b)$ contains the subword $u$.  If $b$ occurs at least twice in $w$ then there is some word $u'$ such that $bu'b$ occurs as a prefix of $w$ and so
$\phi^n(b)\phi^n(u')\phi^n(b)$ is also a prefix of $w$.  But this now means that $uu''u$ occurs as a subword of $w$ and this contradicts the fact that no element from $v\Sigma^* v'$ occurs as a subword of $w$.   Thus if $A_w$ is not prime then the first letter of $w$ occurs only once.  

On the other hand if $b$ is the first letter of $w$ and it only occurs once in $w$ then $bub$ cannot be a subword of $w$ for any word $u$ and so the image of $b$ in $A_w$ generates a nilpotent ideal.   This proves (1) and (2).  

Next, by \cite[Theorem 3.2]{Belov}, $A_w$ is just infinite if and only if $w$ is a uniformly recurrent word, and so this gives (3).  To see (4), note that if $w$ is eventually periodic, then $A_w$ has Gelfand-Kirillov dimension one by Theorem \ref{thm: GK} and hence is PI \cite{SSW}. 
Conversely, suppose that $A_w$ satisfies a polynomial identity.   If we take the ideal $I$ of $A_w$ generated by the images of all subwords of $w$ that only appear finitely many times in $w$ then by construction the algebra $B=A_w/I$ is a prime monomial algebra and satisfies a polynomial identity.  
It is straightforward to see that there is some recurrent word $u$ such that the images of the subwords of $u$ in $B$ form a basis for $B$.   By \cite[Remark, page 3523]{Belov}, we then have that $u$ is periodic since $B$ satisfies a polynomial identity.  We let $q(n)$ denote the collection of subwords of $w$ of length $n$ that appear infinitely often in $w$.  Then by construction $q(n)$ is precisely the number of distinct subwords of $u$ of length $n$ and so $q(n)={\rm O}(1)$.  In particular, there is some $d$ such that $q(d)\le d$.   Let $v_1,\ldots ,v_m$ denote the distinct subwords of $w$ of length $d$.  By assumption, at most $d$ of these words occur infinitely often and so we may write $w=vw'$ where $v$ is finite and $w'$ has at most $d$ subwords of length $d$.  We then have that $w'$ is eventually periodic \cite[Theorem 10.2.6]{AS} and so $w$ is eventually periodic.  This gives (4).


Finally, by \cite[Corollary 5.40]{Belov}, $A_w$ is noetherian if and only if $A_w$ has GK dimension one.  But this occurs if and only if $w$ is eventually periodic by Theorem \ref{thm: GK}.
\end{proof}
\section{Decidability}
One of the advantages of working with iterative algebras is that many of the ring theoretic properties described in the preceding section are decidable for these algebras; that is, one can give an algorithm which inputs the data associated with the morphism and tells whether or not the corresponding algebra has one of the ring theoretic properties described earlier.  As Theorem \ref{thm: char} shows, many of the relevant properties for iterative algebras can be described in terms of a handful of properties of the corresponding pure morphic word.  In particular, it suffices to understand whether a pure morphic word is eventually periodic, uniformly recurrent, or whether certain letters occur twice.  In fact, these are well-understood for pure morphic words and we give a summary of some of what is known.
\begin{mythm} Let $w$ be a pure morphic word over a finite alphabet $\Sigma$.  Then it is decidable whether or not $w$ has the following properties:
\begin{enumerate}
\item $w$ is eventually periodic;
\item $w$ is primitive;
\item $w$ is uniformly recurrent;
\item $a\in \Sigma$ occurs at least twice in $w$;
\end{enumerate}
\end{mythm}
\begin{proof} Decidability of the eventual periodicity question was answered by Harju and Linna \cite{HL} and independently by Pansiot \cite{Pa}. Durand \cite{Durand} showed that uniform recurrence was decidable.  Primitivity and questions about occurrence of letters are much more straightforward. Let $b$ be the first letter of $w$, let $M(w)$ be the incidence matrix of $w$ and let $\phi$ be the underlying morphism.  Then there is some fixed vector ${\bf u}$ such that the number of occurrences of $a$ in $\phi^n(b)$ is given by
$f(n)={\bf u}^T M(w)^n\theta(b)$.  Thus if $a$ occurs at most once, then $f(n)=0,1$ for all $n$.  Now $f(n)$ satisfies a linear recurrence that can be computed explicitly using the Cayley-Hamilton theorem and by computing the first $d$ terms of $f(n)$, where $d=\#\Sigma$.  Solving the recurrence, it is straightforward to decide whether or not this is the case.  Primitivity can be decided analogously.  
\end{proof}
\section{Construction}
In this section, we give an example of an iterative algebra, which gives answer to several questions of Greenfeld, Leroy,
Smoktunowicz, and Ziembowski \cite{GLSZ}.  (In particular, Question 31 and 36 have the answer `no' and Question 32 has the answer `yes'.)   The questions we consider deal with \emph{graded nilpotent} algebras.  These are graded algebras $A$ with the property that if $S$ is a subset of $A$ consisting of homogeneous elements of the same degree, then there is some natural number $N=N(S)$ such that $S^N=(0)$.  We will show how one can use iterative algebras to construct such rings.  Iterative algebras are unital and thus will not be graded nilpotent without some alteration; we will show, however, that by taking the positive part of an iterative algebra then in certain cases one can find a grading that gives a graded nilpotent algebra.  Questions 31 and 32 of \cite{GLSZ} ask respectively whether a graded nilpotent algebra must be Jacobson radical and whether it can have GK dimension two; Question 36 asks whether a graded nilpotent algebra that is finitely generated as a Lie algebra must be nilpotent.  We give answers to these questions with a single example.  Specifically, we construct a graded nilpotent algebra of quadratic growth that is finitely generated as a Lie algebra and whose Jacobson radical is trivial.  
  
Let $\Sigma=\{x_1,\ldots, x_6,y_1,\ldots ,y_6\}$ and let $\phi : \Sigma^*\to \Sigma^*$ be the morphism given by
\[ \left. \begin{matrix} x_1\mapsto x_1x_2y_1y_2& x_2\mapsto x_1x_3 y_1y_3 & x_3 \mapsto x_1x_4y_1y_4
\\
x_4\mapsto x_1x_5y_1y_5 & x_5\mapsto x_1x_6y_1y_6 & x_6\mapsto x_2x_3y_2y_3 \\
y_1\mapsto x_2x_4y_2y_5 & y_2\mapsto x_2x_5y_3y_4 & y_3\mapsto x_2x_6y_2y_6 \\
y_4\mapsto x_3x_4y_3y_5 & y_5 \mapsto x_3x_5y_3y_6 & y_6\mapsto x_3x_6y_4y_5. \end{matrix}\right. \]
Let $w$ be the unique right infinite word whose first letter is $x_1$ and that is a fixed point of $\phi$.  We note that $w$ is $4$-uniform. 

Then a straightforward computer computation shows that the incidence matrix $M(w)$ 
has characteristic polynomial
\begin{equation}
P_w(x):=x^{12}   - x^{11}   - 8 x^{10}   - 16 x^9  - 2 x^8  + 5 x^7  + 5 x^6  + 21 x^5  + 31 x^4  - 10 x^3  - 8 x^2.
\label{eq: charpoly}
\end{equation}
We now put a grading on $A_w$ by declaring that $x_1$ has degree one and that all other letters in $\Sigma$ have degree two.  We let $W_n$ denote the degree of $\phi^n(x_1)$.   Then Equation (\ref{eq: theta}) gives 
\begin{equation}
W_n={\bf u} M(w)^n \theta(x_1),
\end{equation} where ${\bf u}=[1,2,\ldots ,2]$.   A straightforward computer calculation shows that
for $n=0,1,\ldots ,12$, we get the sequence of values
\begin{equation}\label{eq: list}
1,9,40, 162,655,2627,10487,41987,167922, 671648, 2686840, 10746875, 42987905
\end{equation} for $W_n$.

\begin{prop} Let $d$ be a positive integer.  Then with the grading described above, the algebra $A_w$ has the property that the homogeneous component of degree $d$ is nilpotent.
\label{prop: grnilpotent}
\end{prop}
\begin{proof}
We let $\mc{S}=\{s_0,s_1,s_2,\ldots\}$ denote the subset of $\mathbb{N}_0$ constructed as follows.
We define $s_0=0$ and for $i\ge 1$, we define $s_i=s_{i-1}+{\rm deg}(a_i)$, where $a_i\in\Sigma$ is the $i$-th letter of $w$.  Then since $w=x_1x_2y_1y_2x_1x_3y_1y_3\cdots $, we see that
$\mc{S}=\{0,1,3,5,7,8,10,12,14,\ldots\}$.  We now let $(A_w)_d$ denote the homogeneous component of $A_w$ of degree $d$.  Since $(A_w)_d$ is spanned by subwords of $w$ of degree $d$, we see that if $(A_w)_d$ is not nilpotent, then for every $r\ge 1$ there must exist some subword of $w$ of the form $u_1u_2\cdots u_r$ where $u_1,\ldots ,u_r$ are words of degree $d$.  In particular, $\mc{S}$ must contain an arithmetic progression of the form 
$a,a+d,a+2d,\ldots ,a+(r-1)d$ for every $r\ge 1$.  

We now show that this cannot occur.  We suppose, towards a contradiction, that there is some $d\ge 1$ such that $\mc{S}$ contains arbitrarily long arithmetic progressions of the form
$$
\lbrace a, a+d, a+2d, \ldots, a+bd\rbrace,
$$
where $a,b\in \N$. Then for every $n\ge 2$ there exists a subword $u=u_1u_2\cdots u_{r}$ of $w$ such that each $u_i$ is a subword of $w$ of degree $d$ and $r>4^{n+1}$.  Now $u$ has length at least $4^{n+1}$ and since $w$ is a fixed point of $\phi$, we then see that $u$ must have a subword of the form $\phi^n(z)$ for some $z\in \Sigma$.  It follows that 
there exist natural numbers $i$ and $j$ with $i>1$ and $i+j<r$ such that $\phi^n(z)=v u_i\cdots u_{i+j} v'$, where $v$ is a proper (possibly empty) suffix of $u_{i-1}$ and $v'$ is a proper (possibly empty) prefix of $u_{i+j+1}$.  Every $z\in \Sigma$ has the property that $\phi^2(z)$ begins with $x_1$ and so $\phi^n(z)$ begins with $\phi^{n-2}(x_1)$.  In particular, since $\phi^{n-2}(x_1)$ is a prefix of $w$ of length $4^{n-2}$ and it begins $vu_iu_{i+1}\cdots$, we see that $\mc{S}$ contains a progression of the form $t,t+d,t+2d,\ldots ,t+sd$ with $t<d$ and $i+(s+1)d$ strictly greater than the length of $\phi^{n-2}(x_1)$.  By assumption, $\mc{S}$ contains arbitrarily long arithmetic progressions of length $d$ and so there are infinitely many natural numbers $n$ for which 
$\phi^n(x_1)$ contains a progression of the form $t,t+d,t+2d,\cdots t+sd$ for some $t<d$ and $t+(s+1)d$ strictly greater than the length of $\phi^n(x_1)$.  In particular, there is some fixed $t<d$ for which there are infinitely many natural numbers $n$ with this property.  Thus we see $\mc{S}$ contains arbitrarily long arithmetic progressions of the form $t,t+d,t+2d,\ldots ,t+rd$ for some fixed $t<d$.  It follows that $\mc{S}$ contains an infinite arithmetic progression $t,t+d,t+2d,\ldots $.  

Now $$\phi^{n+2}(x_1)=\phi^{n+1}(x_1x_2y_1y_2)=\phi^{n+1}(x_1)\phi^n(x_1x_3y_1y_3)\phi^{n+1}(y_1y_2).$$
Thus $\phi^{n+1}(x_1)\phi^n(x_1)$ is a prefix of $w$ for every $n\ge 1$.

Without loss of generality $d\in \N$ is minimal with respect to having the property that $\mc{S}$ has an infinite arithmetic progression of the form $a+d\N$, where $a< d$. Define
$$
T:=\lbrace a: 0\leq a< d, \{a+dn\colon n\ge 0\}\subseteq  \mc{S}\rbrace.
$$
We write $T=\lbrace i_1, i_2, \dots, i_{p}\rbrace,$ where $i_1<i_2<\dots<i_{p}$. Notice that there exists a positive integer $N$ such that if $j\in \{0,\ldots ,d-1\}\setminus T$ then there is some $m$, depending upon $j$, such that $j+md\le N$ and $j+md\not\in \mc{S}$.

Now let $a_j:=i_{j+1}-i_{j}$ for $j=1,\ldots ,p$, where we take $i_{p+1}=i_1+d$. Consider 
$$
{\bf a}:=(a_1,a_2,\dots, a_p)\in \N^p.
$$
Let $\sigma=(1,2,3,\dots,p)\in S_p$, the symmetric group on $p$ letters. For $\pi\in S_p$, we let  $\pi({\bf a})$ denote $(a_{\pi(1)}, a_{\pi(2)},\dots, a_{\pi(p)})$.

We claim that no non-trivial cyclic permutation of ${\bf a}$ can be equal to ${\bf a}$.  Assume, towards a contradiction, that $\sigma^m(a)=a$ for some $m\in \{1,\ldots ,p-1\}$. Let $\pi=\sigma^m$ so that we then have that $a_i=a_{\pi(i)}$. We see that, by definition, $\mc{S}$ contains the set
$$
\{i_1,i_2,\dots, i_p, i_1+d,i_2+d,\dots, i_p+d, i_1+2d,\dots\}
$$ 
Moreover, the differences between successive terms of this sequence are given by the sequence
$$
(a_1,a_2,\dots, a_p, a_1,a_2,\dots, a_p,a_1,\dots).
$$
By repeatedly applying the identity $\sigma^m(a)=a$, we have that
$$
(a_1,a_2,\dots, a_p, a_1,a_2,\dots, a_p,a_1,\dots)=(a_1,a_2,\dots, a_{m},a_1,a_2,\dots, a_{m},a_1,\dots).
$$
Therefore $\mc{S}$ contains an infinite arithmetic progression of the form $j+(a_1+\dots +a_{m})\N$, and $a_1+\cdots +a_m<d$.  Moreover, by the argument we used earlier, where we observed that $\phi^n(x_1)$ occurs infinitely often in $w$ for every $n$, we see that we can take $j<a_1+\cdots +a_m$. But this contradicts the minimality of $d$. Therefore $\sigma^m(a)\neq a$ for any $m\in \{1,\ldots ,p-1\}$, as claimed.

Now for every $n\ge 2$, we have $\phi^{n+1}(x_1)\phi^n(x_1)$ is a prefix of $w$.  Moreover, by assumption the infinite arithmetic progressions in $\mc{S}$ with difference $d$ all appear in the subset
$$X:=\{i_1,i_2,\ldots ,i_p,i_1+d,i_2+d,\ldots \}.$$
We recall that we let $W_m$ denote the weight of $\phi^m(x_1)$ for each $m\ge 0$.  Then given a positive integer $n> N$, there exists a unique $s\in \{1,\ldots ,p\}$ and a unique $r\ge 0$ such that $i_s+dr$ is the largest positive integer in the set $\{i_q+d \ell \colon 1\le q\le p, \ell\ge 0\}$ that is less than or equal to $W_n$.  Since $\phi^n(x_1)\phi^{n-1}(x_1)$ is a prefix of $w$, we see that the part of the set
$$\{i_1,i_2,\ldots ,i_p, i_1+d,i_2+d\ldots ,i_s+dr,W_n+i_1, W_n+i_2,\ldots , W_n+i_p,W_n+i_1+d,\ldots \}$$
in 
$[0,W_{n}+W_{n-1}]$ is entirely contained in $\mc{S}$.

For $j=1,\ldots ,p$, define $i_{s+j}$ to be $i_{s+j-p}+d$ if $s+j>p$.  Then by definition of $T$, we have that
$\{i_{s+1}+dr, i_{s+2}+dr,\ldots , i_{s+p}+dr,i_{s+1}+d(r+1),\ldots \}\cap [0,W_{n}+W_{n-1}]\subseteq \mc{S}$ and so subtracting $W_n$, the weight of $\phi^n(x_1)$, and using the fact that the prefix $\phi^n(x_1)$ in $w$ is then followed by $\phi^{n-1}(x_1)$ in $w$, we see that 
$$\{i_{s+1}+dr-W_n,i_{s+2}+dr-W_n,\ldots\}\cap [0,W_{n-1}]\subseteq \mc{S}.$$
Since $n-1\ge N$ and $i_{s+j}+dr-W_n\in \{0,\ldots ,d-1\}$ for $j=0,\ldots ,p$, we see from tour choice of $N$ that we must have $i_{s+j}+dr-W_n = i_j$ for $j=1,\ldots ,p$.  
Notice also that $(i_{s+j+1}+dr-W_n)-(i_{s+j}+dr-W_n)=a_{s+j}$ for $j=1,\ldots ,p$, where we take $a_{s+j}=a_{s+j-p}$ if $s+j>p$.  Since 
$i_{s+j}+dr-W_n = i_j$, we see that $a_j=i_{j+1}-i_{j}=a_{s+j}$ for $j=1,\ldots ,p$.  Since no non-trivial cyclic permutation of ${\bf a}$ is equal to ${\bf a}$, we have that $s=p$ and so taking $j=1$ in the equation $i_{s+j}+dr-W_n = i_j$ gives
$i_1+dr+d-W_n = i_1$.  In particular, $W_n\equiv 0~(\bmod~d)$ for all $n> N$.  

We now show that there is no $d>1$ such that $d|W_n$ for all sufficiently large $n$, and we will then obtain the desired result since $\mc{S}$ cannot contain an infinite arithmetic progression with difference one.  

Using Equation (\ref{eq: charpoly}) and the Cayley-Hamilton theorem, we have that 
$W_n$ satisfies the recurrence
\begin{eqnarray}
&~& W_n - W_{n-1} - 8W_{n-2} - 16 W_{n-3} - 2W_{n-4} + 5 W_{n-5} + 5W_{n-6} +21 W_{n-7}\nonumber \\
&~& +31W_{n-8} -10W_{n-9} - 8 W_{n-10}=0 \label{eq: recurrence}
\end{eqnarray} for $n\ge 12$.
In particular, $W_n \equiv W_{n-1}+ W_{n-5} + W_{n-6} +W_{n-7} +W_{n-8} ~(\bmod\, 2)$ for all $n\ge 12$.
We can now show that there are infinitely many $n$ for which $W_n$ is not divisible by $2$.  To see this, suppose that this were not the case.  Then there would exist some largest natural number $\ell$ such that $W_{\ell}$ is odd.  From Item (\ref{eq: list}), we see that $\ell\ge 4$ and so $\ell+8\ge 12$.  But now Equation (\ref{eq: recurrence}) gives that
$W_{\ell+8} + W_{\ell+7}+ W_{\ell+3} + W_{\ell+2} +W_{\ell+1} \equiv W_{\ell} ~(\bmod\, 2)$, which is a contradiction since the left-hand side is even and the right-hand side is odd.  It follows that if $d|W_n$ for all sufficiently large $n$ then $d$ must be odd.  Now suppose towards a contradiction that $d$ is odd and that $d|W_n$ for all sufficiently large $n$ and that $d>1$.  Then there is some odd prime $p$ divides $d$.  Then we have $p|W_n$ for all $n$ sufficiently large.  Let $\ell$ be the largest natural number for which $p$ does not divide $W_{\ell}$.  Then 
since $\gcd(W_4,W_5)=1$, we see that $\ell\ge 4$.  But now, since $\ell+10\ge 12$, Equation (\ref{eq: recurrence}) gives that
$8W_{\ell}$ is a $\mathbb{Z}$-linear combination of elements of the form $W_{\ell+i}$ with $i=1,\ldots ,10$ and so we see that if $p|W_n$ for $n>\ell$ then $p|8W_{\ell}$.  Since $p$ is odd and $p$ does not divide $W_{\ell}$ we get a contradiction.  It follows that $d=1$ and so $\mc{S}$ must contain the progression $\{0,1,2,3,\ldots\}$, but this is clearly false.  The result follows.
\end{proof}
We next show that the algebra $A_w$ is finitely generated as a Lie algebra.
\begin{lemma} Let $v$ be a subword of $w$ of length at least two.  Then some cyclic permutation of $v$ is not a subword of $w$.
\label{lem: cyclic}
\end{lemma}
\begin{proof} Let $d$ denote the length of $v$.  We prove this by induction on $d$.  Our base cases are when $d=2,3,4$.  We consider each of these cases separately.  In each case, we suppose towards a contradiction that there exists a word $v$ of length $d$ all of whose cyclic permutations are subwords of $w$ and we derive a contradiction. 
\vskip 2mm
\emph{Case I: d=2}.  For the case when $d=2$, observe that a subword of $w$ of length two is either of the form $x_ix_j$ with $i<j$ or $y_iy_j$ with $i<j$ or of the form $x_k y_{\ell}$ or $y_{\ell}x_k$.  It is immediate that if our subword of length two is of the form $x_i x_j$ or $y_i y_j$ then $i<j$ and so $x_j x_i$ and $y_j y_i$ cannot be subwords of $w$ in this case.  Thus for the case when $d=2$, it only remains to show that if $x_i y_j$ is a subword of $w$ then $y_j x_i$ cannot be.  Observe that any subword of $w$ of length two is either a subword of $\phi(a)$ for some $a\in \Sigma$ or it is a subword of $\phi(ab)$, with $a,b\in \Sigma$, consisting of the last letter of $\phi(a)$ followed by the first letter of $\phi(b)$.  In the case that we have a word of the form $x_i y_j$, then we see that it must be the second and third letters of $\phi(a)$ for some $a\in \Sigma$.  We are also assuming that $y_jx_i$ is a subword of $w$.  In this case, we have that there are letters $b,c\in \Sigma$ such that $bc$ is a subword of $w$ and $y_j$ is the last letter of $\phi(b)$ and $x_i$ is the first letter of $\phi(c)$.  

Since the second letter of any $\phi(a)$ must be in $\{x_2,x_3,x_4,x_5,x_6\}$, we see that $i\neq 1$.  Similarly, the first letter of any $\phi(c)$ must be in $\{x_1,x_2,x_3\}$ and so $i\neq 4,5,6$.  Thus $i\in \{2,3\}$.  Notice that the last letter of $\phi(b)$ can never be $y_1$ and so $j\in \{2,3,4,5,6\}$.  By assumption, there exists some $a\in \Sigma$ such that $x_iy_j$ is a subword of $\phi(a)$ with $i\in \{2,3\}$ and $j>1$. Looking at the map $\phi$, we see that the only possibility is $i=3$, $j=2$ coming from  $a=x_6$.  But then $y_jx_i = y_2 x_3$.  By assumption, there exist $b$ and $c$ in $\Sigma$ such that $bc$ is a subword of $w$ and such that $y_2$ is the last letter of $\phi(b)$ and $x_3$ is the first letter of $\phi(c)$.  But $x_1$ is the only letter in $\Sigma$ with the property that applying $\phi$ to it gives a four-letter word ending in $y_2$.  Thus $b=x_1$.  Similarly, since $\phi(c)$ begins with $x_3$ we see $c\in \{y_4,y_5,y_6\}$.  But this then means that $x_1y_k$ is a subword of $w$ for some $k\ge 4$, which is impossible since $x_1$ is always followed by an element from $\{x_2,\ldots ,x_6\}$ in $w$. 
\vskip 2mm
\emph{Case II: d=3}. Since we can never have three consecutive letters from $\{x_1,\ldots ,x_6\}$ occurring in $w$, we see that there are at most two letters from $x_1,\ldots ,x_6$ in $v$.  Similarly, there are at most two letters from $y_1,\ldots ,y_6$ occurring in $v$.  Thus either $v$ has exactly two letters from $\{x_1,\ldots ,x_6\}$ and one letter from $\{y_1,\ldots ,y_6\}$ or it has exactly two letters from 
$\{y_1,\ldots ,y_6\}$ and one letter from $\{x_1,\ldots ,x_6\}$.  We consider the first case, as the other case is identical.  In the first case, $v$ has some cyclic permutation of the form $x_i y_j x_k$, which cannot be a subword of $w$ since we must always have a block of two consecutive $y_j$'s between blocks of $x_i$'s.  
\vskip 2mm
\emph{Case III: d=4.} In this case, we can argue as in Case II to show that some cyclic permutation of $v$ is of the form $x_i x_j y_k y_{\ell}$.  Any subword of $w$ of the form $x_i x_j y_k y_{\ell}$ must be $\phi(a)$ for some $a\in \Sigma$.  In particular, $i<j$ and $k<\ell$.  But by assumption $y_k y_{\ell}x_i x_j$ is also a subword of $w$ and so there exist $b,c\in \Sigma$ such that $bc$ is a subword of $w$, $y_k y_{\ell}$ are the last two letters of $\phi(b)$ and $x_ix_j$ are the last two letters of $\phi(c)$.   
But the first two letters of $\phi(d)$ for $d\in \Sigma$ completely determine $d$, and similarly for the last two letters.  In particular, $b=c=a$ and so $bc=a^2$ is a subword of $w$.  But we now see this is impossible from Case I.  

\vskip 2mm
We now complete the induction argument.  Suppose that $d\ge 5$ is such that if $m\in \{2,\ldots ,d-1\}$ then any subword $u$ of $w$ of length $m$ has the property that some cyclic permutation of $u$ is not a subword of $w$.  Let $v$ be a word of length $d$.  Then arguing as in Case II, we have that some cyclic permutation $v'$ of $v$ is of the form
$x_{i_1}x_{i_2}y_{j_1}y_{j_2}x_{i_3}\cdots $.  Notice that if $d$ is $1$ or $2~(\bmod 4)$ then $v'$ ends with some $x_k$ and so if we move this $x_k$ to the beginning of $v'$ then we get a cyclic permutation of $v$ with three consecutive letters $x_k x_{i_1}x_{i_2}$, which cannot be a subword of $w$.  If $d$ is $3~(\bmod 4)$, then the last two letters of $v'$ are of the form $x_i y_j$ and so if we shift these two letters to the beginning of $v'$ we see that $v$ has a cyclic permutation with three consecutive letters $x_iy_jx_{i_1}$, which cannot be a subword of $w$.  Thus we have $d=4m$ with $m\ge 2$.  Also
$v'=\phi(a_1)\cdots \phi(a_m)$ for some $a_1,\ldots ,a_m\in \Sigma$.  But then if $u=a_1\ldots a_m$ and $u'$ is a cyclic permutation of $u$ then $\phi(u')$ is a cyclic permutation of $\phi(u)=v'$.  Moreover, since $\phi(u')$ begins with two consecutive letters from $x_i$ and $x_j$, we see that $\phi(u')$ is a subword of $w$ if and only if $u'$ is a subword of $w$.  In particular, if all cyclic permutations of $v$ are subwords of $w$ then all cyclic permutations of $u$ are subwords of $w$.  But $u$ has length $m\in \{2,\ldots ,d-1\}$ and so we see that this cannot occur by the induction hypothesis.  The result now follows.
\end{proof}
\begin{prop} The algebra $A_w$ is finitely generated as a Lie algebra.\label{prop: fg}\end{prop}
\begin{proof} Let $B$ denote the Lie subalgebra of $A_w$ that is generated by the elements
$$\{1,x_1,\ldots ,x_6,y_1,\ldots ,y_6\}.$$  Since $A_w$ is spanned by the images of all subwords of $w$, it suffices to show that the image of every subword of $w$ is in $B$.  Suppose that $u$ is a subword of $w$ whose image is not in $B$.  We may assume that we pick $u$ of minimal length $d$ with respect to having this property.  Since all subwords of $w$ of length $\le 1$ have images in $B$, $d$ must be
 at least two.  By Lemma \ref{lem: cyclic}, $u$ has some cyclic permutation that is not a subword of $w$.  In particular, we may decompose $u=ab$ so that $a$ and $b$ are subwords of $w$ but such that $ba$ is not a subword of $w$; i.e., $ba$ has zero image in $A_w$.  Then $u=[a,b]\in A_w$.  But now $a$ and $b$ have length less than $d$ and so by minimality of $d$, we have that the images of $a$ and $b$ are in $B$ and so the image of $u=[a,b]$ is in also in $B$, a contradiction.  The result follows.
 \end{proof}
We need one last result to obtain our example.  
\begin{lemma} \label{lem: evper}
The word $w$ is uniformly recurrent.
\end{lemma}
\begin{proof}
It is straightforward to show that $\phi^2(u)$ begins with $x_1$ for each $u\in \Sigma$.  Since $w$ can be written as 
$\phi^2(u_1)\phi^2(u_2)\cdots$ with $u_1,u_2,\ldots \in \Sigma$, and since each $\phi^2(u_i)$ has length $16$, we see that each subword of $w$ of length at least $16$ contains $x_1$.  
Let $v$ be a subword of $w$ and let $n$ be a natural number.
Write $w=a_1a_2\cdots$ with $a_i\in \Sigma$.  Then 
we have shown that whenever $a_i=x_1$ there is some $N\le 16$ such that $a_{i+N}=x_1$.  Since
$w=\phi^n(a_1)\phi^n(a_2)\cdots$ and each $\phi^n(a_j)$ has length $4^n$, we see that if $\phi^n(x_1)$ occurs in $w$ then there is another occurrence of $\phi^n(x_1)$ beginning at most $16\cdot 4^n$ places later.  In particular, since each subword of $w$ is a subword of $\phi^n(x_1)$ for some $n$, we see that $w$ is uniformly recurrent.
\end{proof}
Putting these results together, we obtain the following result, which answers questions 31, 32, and 36 of \cite{GLSZ}.
  
  \begin{mythm}\label{thm: main2} Let $R$ denote the positive part of the algebra $A_w$ with the grading described above.  Then $R$ is just infinite, graded nilpotent, has quadratic growth, has trivial Jacobson radical, and is finitely generated as a Lie algebra.\end{mythm}
\begin{proof}
The fact that $R$ is graded nilpotent follows from Proposition \ref{prop: grnilpotent}.  By Lemma \ref{lem: evper} and Theorem \ref{thm: char} (4), we see that $A_w$ is just infinite and hence $R$ is just infinite.  
To prove the remaining claims, we first show that $A_w$ does not satisfy a polynomial identity.  To see this, by Theorem \ref{thm: char} it suffices to show that $w$ is not eventually periodic.  Towards a contradiction, suppose that $w$ is eventually periodic.  Then by Theorem \ref{thm: char}, $A_w$ is an infinite-dimensional graded noetherian PI algebra.  It is also prime since $A_w$ is just infinite. But now the graded version of Goldie's theorem \cite{GS} gives that $A_w$ must have a regular homogeneous element of positive degree, a contradiction.  Thus $A_w$ is not PI.  

Since $w$ is not eventually periodic, we now see that $R$ has quadratic growth by Proposition \ref{prop: quadratic}, since $w$ is $4$-uniform.  
To show that $R$ has trivial Jacobson radical, it suffices to show that $A_w$ is semiprimitive.   Since $w$ is uniformly recurrent and $w$ is not eventually periodic, Corollary 3.11 and Proposition 3.8 of \cite{Belov} then gives that the Jacobson radical of $A_w$ is zero.  Finally, $R$ is finitely generated as a Lie algebra by Proposition \ref{prop: fg}. This completes the proof.  
\end{proof}
In fact, since $A_w$ is just infinite, it is prime and so by a result of Okni{\'n}ski \cite{Ok}, the algebra $A_w$ must be primitive, since it is not PI and has zero Jacobson radical. 
\section{Concluding questions}
There are lots of natural questions one can ask about iterative algebras.  We restrict ourselves to posing just a small number of questions, which we have not investigated in this article, but which we feel could be of independent interest or of later use.
The first, and most compelling question, has to do with understanding ${\sf QGr}(A_w)$, the category of $\mathbb{Z}$-graded right $A_w$-modules modulo the full subcategory consisting of modules that are the sum of finite-dimensional submodules.  Holdaway and Smith \cite{HS} investigated this category for finitely presented algebras, and it is natural to investigate it in this setting.
\begin{quest} Can one give a concrete description of the category ${\sf QGr}(A_w)$ for an iterative algebra $A_w$?
\end{quest}
A second question deals with when, precisely, can an iterative algebra be given a grading that makes the algebra graded nilpotent.
\begin{quest} Can one determine when an iterative algebra associated to a pure morphic word $w$ over an alphabet $\Sigma$ has a grading, induced by making the letters of $\Sigma$ homogeneous of some positive degrees, such that the positive part of $A_w$ is graded nilpotent?
\end{quest}
\begin{quest} Can one determine for which pure morphic words $w$ the Jacobson radical of $A_w$ is trivial?
\end{quest}
By a result of Okni{\'n}ski \cite{Ok}, a prime monomial algebra is either primitive, satisfies a polynomial identity, or has a nonzero locally nilpotent Jacobson radical.  Thus if one could decide whether the Jacobson radical is nonzero then one can decide whether or not the algebra is primitive by Theorem \ref{thm: char}.

We are also intrigued by the fact that there are five different types of growth of iterative algebras, as shown by Theorem \ref{thm: GK}.  We pose the following question.
\begin{quest} Can one give purely ring theoretic characterizations that determine which of the five possible growth types an iterative algebra has?
\end{quest}
This question is admittedly vague, but we observe that $A_w$ has linear growth if and only if $A_w$ satisfies a polynomial identity.  We wonder if, in a similar vein, one can find ring theoretic characterizations for the other growth types.

For the last question, we note that decidability of whether a morphic word is eventually periodic is a longstanding open problem (cf. \cite[p. 2]{HHKR}).  In terms of algebras this corresponds to looking at graded subalgebras (generated in degree one) of $A_w$ and asking whether or not one can decide if the subalgebra has Gelfand-Kirillov dimension one.  Motivated by these problems from theoretical computer science, we pose the following questions. 
\begin{quest} Given an iterative algebra $A_w$ can one determine the possible values for the GK dimension of a subalgebra $B$ that is generated by elements of degree one?   Is it decidable whether or not $B$ has GK dimension one?
\end{quest}
Finally, there are interesting questions about finite presentation and Hilbert series of iterative algebras.
\begin{quest} Given an iterative algebra $A_w$ can one characterize the words for which $A_w$ has rational Hilbert series with the standard grading?  Can one characterize the words for which $A_w$ is finitely presented?
\end{quest}

\end{document}